\documentclass[12pt]{amsart}
\usepackage[foot]{amsaddr}
\usepackage[mathscr]{eucal}
\usepackage{geometry}
\usepackage{amsmath,amsfonts,amssymb,amsthm,url}
\usepackage{MnSymbol}
\usepackage{verbatim} 
\usepackage{hyperref}
\usepackage[english]{babel}
\usepackage{graphicx}
\usepackage[enableskew]{youngtab}
\newtheorem{thm}{Theorem}[section]
\newtheorem{conj}[thm]{Conjecture}
\newtheorem{cor}[thm]{Corollary}
\newtheorem{lem}[thm]{Lemma}
\newtheorem{prop}[thm]{Proposition}

\newtheorem{innercustomcor}{Corollary}
\newenvironment{customcor}[1]
  {\renewcommand\theinnercustomcor{#1}\innercustomcor}
  {\endinnercustomcor}
\makeatletter

\newtheorem*{repeat@theorem}{\repeat@title}
\newcommand{\newrepeattheorem}[2]{%
\newenvironment{repeat#1}[1]{%
 \def\repeat@title{#2 \ref{##1}}%
 \begin{repeat@theorem}}%
 {\end{repeat@theorem}}}
\makeatother

\newrepeattheorem{conj}{Conjecture}
\theoremstyle{definition}
\makeatletter
\newtheorem*{rep@theorem}{\rep@title \ continued}
\newcommand{\newreptheorem}[2]{%
\newenvironment{rep#1}[1]{%
 \def\rep@title{#2 \ref{##1}}%
 \begin{rep@theorem}}%
 {\end{rep@theorem}}}
\makeatother

\newreptheorem{example}{Example}
\theoremstyle{definition}

\newtheorem{defn}[thm]{Definition}
\newtheorem{example}[thm]{Example}
\newtheorem{remark}[thm]{Remark}
\newtheorem{non-example}[thm]{Non-Example}
\newtheorem{problem}{Problem}

\def\a{\alpha}

 \title{Support Equalities Among  Ribbon Schur Functions}
 
 \date{}
 
 \author{Marisa Gaetz$^1$}
 \address{$^1$Massachusetts Institute of Technology}
\email{mgaetz@mit.edu}

 \author{Will Hardt$^2$} 
 \address{$^2$University of Wisconsin -- Madison}
\email{whardt@wisc.edu}

 \author{Shruthi Sridhar$^3$} 
 \address{$^3$Princeton University}
 \email{ssridhar@math.princeton.edu}
 
\begin{document}
\begin{abstract}
In 2007, McNamara proved that two skew shapes can have the same Schur support only if they have the same number of $k\times \ell$ rectangles as subdiagrams. This implies that two ribbons can have the same Schur support only if one is obtained by permuting row lengths of the other. We present substantial progress towards classifying when a permutation $\pi \in S_m$ of row lengths of a ribbon $\alpha$ produces a ribbon $\alpha_{\pi}$ with the same Schur support as $\alpha$; when this occurs for all $\pi \in S_m$, we say that $\alpha$ has \emph{full equivalence class}. Our main results include a sufficient condition for a ribbon $\alpha$ to have full equivalence class. Additionally, we prove a separate necessary condition, which we conjecture to be sufficient.  
\end{abstract}

\maketitle


\section{Introduction}

The question of when two skew diagrams yield equal skew Schur functions has been studied in detail; for instance, see \cite{Billera}, \cite{McWilligenburg}, and \cite{RSvW}. However, the related question of when two skew diagrams have the same Schur support (see Definition \ref{def:ssupport}) has received less attention, with the most substantial progress occurring in \cite{McNamara} (2007) and in \cite{McWilligenburgSupport} (2011). 

In \cite{McNamara}, P. R. W. McNamara proves that any two skew diagrams with the same Schur support necessarily contain the same number of $k \times \ell$ rectangles, for every $k, \; \ell \geq 1$. In \cite{McWilligenburgSupport}, P. R. W. McNamara and S. van Willigenburg explicitly determine the Schur support for a special class of skew shapes called \emph{equitable ribbons}.

In this paper, we expand on the results presented in \cite{McNamara} and \cite{McWilligenburgSupport} by working to classify which ribbons have the same Schur support under all permutations of their row lengths; we say these ribbons have \emph{full equivalence class} (Definition \ref{fullec}). Note that in the work that follows, we always assume that each row of a ribbon is at least two boxes in length and that each ribbon has at least three rows. This is because in the cases where one of these conditions does not hold, it is fairly easy and uninteresting to classify when a ribbon has full equivalence class. 

In the next section, we provide preliminary information to aid the understanding of the rest of the paper. In Section \ref{sec:sufficient} (resp. Section \ref{sec:necessary}), we provide a sufficient (resp. necessary) condition for a ribbon to have full equivalence class. Our sufficient condition is a generalization of \cite[Thm 1.5]{McWilligenburgSupport}, in which McNamara and van Willigenburg show that all equitable ribbons have full equivalence class. Finally, we conjecture that the necessary condition from Section \ref{sec:necessary} is in fact sufficient.

\section{Preliminaries}

We begin by establishing some basic definitions relating to Schur functions and ribbons. We will then introduce the Littlewood-Richardson rule, which is a central ingredient in our proofs. Finally, we will introduce an algorithm given by $R$-matrices, which will dictate a way to swap adjacent row lengths in a ribbon tableau, while preserving the content of the filling and semistandardness within the two swapped rows. This algorithm is an important part of our proof of a sufficient condition for full equivalence class (Corollary \ref{suff-real}).

\subsection{Schur Functions}

The \textit{Young diagram} corresponding to a partition $\lambda = (\lambda_1, \lambda_2, \ldots, \allowbreak \lambda_m)$ of an integer $n$ is a collection of $n$ boxes arranged in left-aligned rows, where the $i^{th}$ row from the top has $\lambda_i$ boxes. A filling of a Young diagram with integers is called \textit{semistandard} if the integers increase weakly across rows and strictly down columns. Such a filled-in Young diagram is called a \textit{semistandard Young tableau (SSYT)} (Figure \ref{young-tab-figure}). 

\begin{figure} 
\begin{center}
\young(11111224,2233455,345)
\end{center}
\caption{A semistandard Young tableau.}
\label{young-tab-figure}
\end{figure}

We use \textit{weight} or \textit{content} to refer to the multi-set of integers in the filling of a tableau. The weight or content is denoted as a tuple $\mu = (\mu_1, \mu_2, \ldots, \mu_k)$, where $\mu_i$ is the number of $i$'s in the filling of the tableau. For example, the content of the tableau from Figure \ref{young-tab-figure} is $\mu = (5,4,3,3,3)$. 

\emph{Schur functions} are often considered to be the most important basis for the ring of symmetric functions. Schur functions are indexed by integer partitions, where the Schur function $s_{\lambda}$ corresponding to a partition $\lambda$ is defined as

\begin{equation} \label{eq:SchurFunc}
s_{\lambda}(x_1, x_2, x_3, \ldots) = \sum_{\substack{T~:~\text{SSYT of}\\ \text{shape } \lambda}} x^T = \sum_{\substack{T~:~\text{SSYT of}\\ \text{shape } \lambda}} x_1^{t_1} x_2^{t_2} x_3^{t_3} \cdots 
\end{equation}
where $t_i$ is the number of occurrences of $i$ in $T$.

We can generalize this notion of Schur functions to apply to skew shapes, which are obtained by removing the Young diagram corresponding to the partition $\mu$ from the top-left corner of a larger Young diagram corresponding to the  partition $\lambda$. Here, we require that the diagram for $\mu$ is contained in the diagram for $\lambda$, and we write the resulting skew shape as $\lambda / \mu$. When $\mu$ is the empty partition, we call $\lambda / \mu$ ``straight." Skew Schur functions have an analogous definition to that of straight Schur functions, where the sum in Equation \ref{eq:SchurFunc} is instead over skew semistandard Young tableaux of shape $\lambda / \mu$.

Skew Schur functions have the nice property that they are Schur-positive, meaning that for any skew shape $\lambda / \mu$, we can write
$$s_{\lambda / \mu} = \sum_{\nu} c_{\mu, \nu}^{\lambda} s_{\nu}$$
where $\nu$ denotes a straight partition, and where all coefficients $c_{\mu, \nu}^{\lambda} \geq 0$. The coefficients $c_{\mu, \nu}^{\lambda}$ are called \emph{Littlewood-Richardson coefficients}, and will play an important role in the Littlewood-Richardson rule (which we introduce in Theorem \ref{LR-rule}). Since such a decomposition of a skew Schur function into a linear combination of straight Schur functions is unique, the following definition is well-defined:

% This relationship between skew Schur functions and straight Schur functions motivates the following definition:

\begin{defn} \label{def:ssupport}
The \emph{Schur support} of a skew shape $\lambda / \mu$, denoted $[\lambda / \mu]$, is defined as
$$[\lambda / \mu] = \{\nu: \; c_{\mu, \nu}^{\lambda} > 0\}.$$
\end{defn}
In other words, the support of a skew shape $\lambda/\mu$ is the set of straight shapes $\nu$ such that $s_{\nu}$ appears with nonzero coefficient in the expansion of $s_{\lambda / \mu}$ into a linear combination of straight Schur functions.

\begin{remark} \label{rmk:antipodal}
It is well known \cite[Exer. 7.56(a)]{Stanley} that $[\a^{\circ}] = [\a]$, where $\a^{\circ}$ is the antipodal ($180^{\circ}$) rotation of a ribbon $\a$.
\end{remark}

\subsection{Ribbons}

A \textit{ribbon} is a connected skew shape which does not contain a $2 \times 2$ block as a subdiagram. Any composition $\alpha$ of an integer $n$ determines a unique ribbon. We will use the notation $\alpha = (\alpha_1, \alpha_2, \ldots, \alpha_m)$ to denote a ribbon with $m$ rows, where row 1 is at the top of the ribbon, row $i$ has length $\alpha_i$, and $\a_i \geq 2$ for all $1 \leq i \leq m$.

\begin{defn}
Let $\a = (\a_1, \a_2, \ldots, \a_m)$ and $\a_{\pi} = (\a_{\pi^{-1}(1)}, \a_{\pi^{-1}(2)}, \ldots, \a_{\pi^{-1}(m)})$ be ribbons, where $\pi \in S_m$. We say $\a_{\pi}$ is a \textit{permutation} of $\a$.  
\end{defn}

\begin{example} \label{ex:perms}
Figure \ref{ex:perms-figure} depicts all permutations of the ribbon $\a = (4,3,2)$, where we have written the permutations in cycle notation.
\begin{figure}
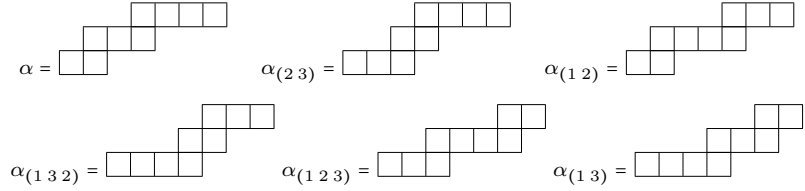

\begin{center}
\tiny$\a =$ \young(:::\;\;\;\;,:\;\;\;,\;\;) \hspace{.25cm} $\a_{(2 \; 3)} = $ \young(:::\;\;\;\;,::\;\;,\;\;\;) \hspace{.25cm} $\a_{(1 \; 2)} = $ \young(::::\;\;\;,:\;\;\;\;,\;\;) \\
\vspace{.25cm}
$\a_{(1 \; 3 \; 2)} =$ \young(::::\;\;\;,:::\;\;,\;\;\;\;) \hspace{-.1cm} $\a_{(1 \; 2 \; 3)} = $ \young(:::::\;\;,::\;\;\;\;,\;\;\;) \hspace{-.1cm} $\a_{(1 \; 3)} = $ \young(:::::\;\;,:::\;\;\;,\;\;\;\;)
\end{center}
\caption{Permutations of the ribbon $\a = (4,3,2)$.}
\label{ex:perms-figure}
\end{figure}
\end{example}
\vspace{.25cm}

In \cite{McNamara}, McNamara proves that two skew shapes can have the same Schur support only if they contain the same number of $k \times \ell$ rectangles as subdiagrams, for all $k, \ell \geq 1$. This result has the following implication for ribbons:

\begin{prop} \label{prop: McNamara}
Let $\alpha$ and $\beta$ be ribbons such that $[\alpha] = [\beta]$. Then $\beta = \alpha_\pi$ for some permutation.
\end{prop}

\begin{proof}
By \cite{McNamara}, $\a$ and $\beta$ contain the same number of $2 \times 1$ rectangles as subdiagrams, and therefore they contain the same number of rows --- let's say $m$ rows. Label the row lengths of $\a$, indexing so that the row lengths weakly decrease as the index increases (i.e. $\a_1 \geq \a_2 \geq \cdots \geq \a_m$). Label the row lengths of $\beta$ in the same way. It suffices to show that $\alpha_i = \beta_i$ for all $1 \leq i \leq m$.

Suppose for the sake of contradiction that there exists an $i \in \{1,2,\ldots,m \}$ for which $\alpha_i \neq \beta_i$. Choose the minimal such $i$ and assume, without loss of generality, that $\alpha_i > \beta_i$. It follows that $\alpha$ has more $1 \times \alpha_i$ rectangles than $\beta$ does, contradicting McNamara's necessary condition for Schur support equality. Therefore, $\alpha_i = \beta_i$ for all $1 \leq i \leq m$, completing the proof. 
\end{proof}

\begin{defn} \label{fullec}
We say that a ribbon $\a=(\a_1, \a_2, \ldots, \a_m)$ has \emph{full equivalence class} if $[\a] = [\a_\pi]$ for all permutations $\pi \in S_m$.
\end{defn}

For instance, the ribbon $\a = (4,3,2)$ from Example \ref{ex:perms} has full equivalence class, since all of its permutations have support $\{(7,2), (7,1,1), (6,3),(6,2,1), (5,4), (5,3,1), \allowbreak (5,2,2), (4,4,1), (4,3,2)\}$.

\begin{defn} \label{triangle-ineqality}
We say integers $x \leq y \leq z$ satisfy the \textit{strict triangle inequality} if $z < x + y$. In this case, we may also say that the set $\{x,y,z\}$ satisfies the strict triangle inequality.
\end{defn}

\subsection{Yamanouchi Words and Tableaux}

We now introduce the concepts of \textit{Yamanouchi words} and \textit{Yamanouchi tableaux}, which will be essential for using and defining our main tool for proving equality of support --- the Littlewood-Richardson rule. 

A \textit{Yamanouchi word} is a word with the property that all of its prefixes contain no more $(i+1)$'s than $i$'s, for all integers $i \geq 1$. For our purposes, we are concerned with the \textit{reverse reading word} (henceforth ``RRW'') of a tableau, which reads right-to-left across rows and top-to-bottom from one row to the next. A \textit{Yamanouchi tableau} is a tableau whose RRW is a Yamanouchi word.

\begin{example} \label{ex: LR} The tableau depicted in Figure \ref{ex: LR-figure} is Yamanouchi because each prefix of its RRW (i.e. of $112213321$) contains no more 2's than 1's, and contains no more 3's than 2's. 
\begin{figure}
\begin{center}
\young(:::::11,:::122,1233)
\end{center}
\caption{Young tableau with RRW 112213321.}
\label{ex: LR-figure}
\end{figure}
\end{example}

\subsection{Littlewood-Richardson Tableaux}

\textit{Littlewood-Richardson tableaux} (which we often abbreviate as \textit{LR-tableaux}) are tableaux which are both semistandard and Yamanouchi. These tableaux play an important role in the Littlewood-Richardson rule, which we are now ready to state.

\begin{thm}[Littlewood-Richardson rule \cite{LiRi}] \label{LR-rule}
If $$s_{\lambda / \mu} = \sum_{\nu} c_{\mu, \nu}^{\lambda} s_{\nu},$$ then the number of Littlewood-Richardson tableaux of shape $\lambda / \mu$ and content $\nu$ is given by $c_{\mu, \nu}^{\lambda}$.
\end{thm}

Although the Littlewood-Richardson rule was originally asserted by Littlewood and Richardson in 1934, they only proved the theorem in certain special cases. The first rigorous proofs of the result were not given until much later, the earliest of which occurred in 1974 and 1976, given by Thomas \cite{Proof1} and Sch{\"u}tzenberger \cite{Proof2}, respectively. For a short proof of the Littlewood-Richardson rule, see \cite{ShortProof}. 

The following corollary follows immediately from Theorem \ref{LR-rule} and the definition of Schur support (Definition \ref{def:ssupport}), and will be more directly applicable to the proofs in the remainder of the paper. In fact, throughout the paper, when we cite ``the Littlewood-Richardson rule,'' we generally are referring to the following corollary of it.

\begin{cor} \label{LR-support}
For any straight shape $\nu$ and skew shape $\lambda / \mu$, we have $\nu \in [\lambda / \mu]$ if and only if there exists a Littlewood-Richardson tableau of shape $\lambda / \mu$ and content $\nu$.
\end{cor}

\begin{repexample}{ex: LR}
Observe that the tableau in Figure \ref{ex: LR-figure} is both semistandard and Yamanouchi, with content $(4,3,2)$. It follows from Theorem \ref{LR-support} that the straight shape $(4,3,2)$ is in the support of the ribbon with row lengths $(2,3,4)$.
\end{repexample}

\subsection{R-Matrices}
We will now introduce an algorithm which will be instrumental in proving our sufficient condition for a ribbon to have full equivalence class. The \textit{R-matrix algorithm}, described in \cite[Section 2.2.3]{RMat}, provides a way to swap two consecutive row lengths in an arbitrary ribbon with a semistandard filling so that the filling within the two rows remains semistandard and has the same content as before. Note, however, that semistandardness of the ribbon as a whole is not necessarily preserved.

Let $A$ be a ribbon tableau of shape $\alpha = (\alpha_1, \ldots, \alpha_m)$, and assume that $\a_j > \a_{j+1}$ (note that if $\a_j < \a_{j+1}$, we can swap $\a_j$ and $\a_{j+1}$ by performing this algorithm on the antipodal rotation $\a^\circ$, and then taking the antipodal rotation of the result). The algorithm proceeds as follows.

\begin{enumerate}
\item Convert the filling of rows $j$ and $j+1$ to a box-ball system with the boxes corresponding to the $j^{th}$ row on the left and the boxes corresponding to the $(j+1)^{st}$ row on the right (see Example \ref{ex:boxball}).
\item For each ball on the right (in any order), connect it to the unconnected ball on the left strictly above it which is as low as possible. If there is no such ball on the left, connect it to the lowest unconnected ball on the left.
\item Shift all unconnected balls on the left horizontally to the right. 
\item Convert this box-ball system back into rows of a ribbon tableau.
\end{enumerate}

\begin{example} \cite[Sect. 2.2.3]{RMat} \label{ex:boxball}
Suppose we have a ribbon whose $j^{th}$ and $(j+1)^{st}$ rows are as in Figure \ref{rows-bef-swap}. Steps 1-3 of the $R$-matrix algorithm as applied to this partial tableau are depicted in Figure \ref{boxball}. Notice that the only ball movement is two balls in the third box from the top shifting from the left to the right, as these were the only two unconnected balls on the left. After applying step 4 of the $R$-matrix algorithm, we obtain the partial tableau depicted in Figure \ref{swaptableau}. Notice that the row lengths have swapped, while the content and semistandardness of the filling has been preserved, as promised.
\end{example}

\begin{figure}
\begin{center}
\young(::::::13347,::::135)
\caption{Rows before swapping.}
\label{rows-bef-swap}
\end{center}
\end{figure}

\begin{figure}
\includegraphics[scale = 0.45]{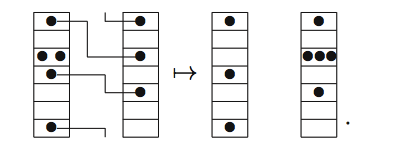}
\centering
\caption{Application of $R$-matrix algorithm.}
\label{boxball}
\end{figure}

\begin{figure}
\begin{center}
\young(::::::::147,::::13335)

\caption{Rows after swapping.}
\label{swaptableau}
\end{center}
\end{figure}

\section{A Sufficient Condition} \label{sec:sufficient}

In this section, we will prove the following sufficient condition for a ribbon to have full equivalence class. 

\begin{customcor}{3.4} \label{suff-real}
Let $\a = (\a_1, \a_2, \ldots, \a_m)$ be a ribbon with each $\a_i \geq 2$ and $m \geq 3$. If all 3-subsets of $\{\a_i\}_{i=1}^m$ satisfy the strict triangle inequality, then $\a$ has full equivalence class.
\end{customcor}

Towards the goal of proving this sufficient condition for a ribbon to have full equivalence class, we begin with two lemmas. In Lemma \ref{R-Yamanouchi}, we prove that the $R$-matrix algorithm preserves the Yamanouchi property of ribbon LR-tableaux. Lemma \ref{R-bot-left} shows that the $R$-matrix algorithm, when applied to a ribbon LR-tableau, preserves some of the semistandardness of the tableau. In particular, if we use the $R$-matrix algorithm to swap rows $j$ and $j+1$ of a ribbon LR-tableau $A$, then semistandardness is preserved between the $(j+1)^{st}$ and $(j+2)^{nd}$ rows of $A$. 

We use the results of these two lemmas to show in Theorem \ref{Tri-Swap-Thing} that under a certain condition on three adjacent row lengths of a ribbon LR-tableau, the bottom two of the three adjacent row lengths can be swapped while preserving the Yamanouchi property and semistandardness of the tableau. By imposing this condition on all rows of the ribbon, we get as a corollary a sufficient condition for a ribbon to have full equivalence class (Corollary \ref{Gen-Tri-Thing}).

\begin{lem} \label{R-Yamanouchi}
Let $A$ be a ribbon LR-tableau of shape $\alpha = (\alpha_1, \ldots , \alpha_m)$ and let $i \in \{1,2,\ldots , m-1\}$. If $A_{(i \; i+1)}$ is the ribbon tableau of shape $\a_{(i \; i+1)}$ that results from applying the $R$-matrix operation to rows $i$ and $i+1$ of $A$, then $A_{(i \; i+1)}$ is a Yamanouchi tableau.
\end{lem}

\begin{proof}
Since $A$ is Yamanouchi, we only need to show that the prefixes of the RRW up through the $i^{th}$ and $(i+1)^{st}$ rows of $A_{(i \; i+1)}$ are Yamanouchi. Consider the $i^{th}$ and $(i+1)^{st}$ rows of $A$ as a subtableau $B$ of $A$. Define $B_{(i \; i+1)}$ analogously with respect to $A_{(i \; i+1)}$. Fix any positive integer $j$ and assume that $j$ and $j+1$ appear in $B$ as shown in Figure \ref{B1}.
\begin{figure}
\begin{center}
\includegraphics[scale=.25]{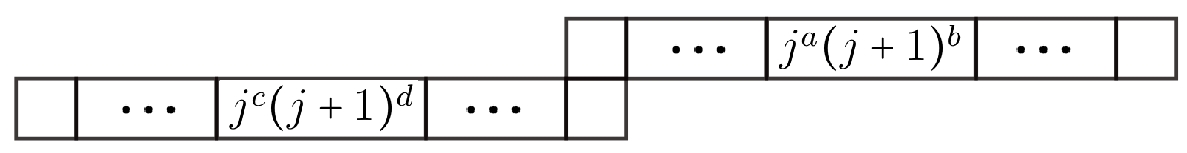}
\end{center}
\caption{Integers $j$ and $j+1$, as they appear in $B$.}
\label{B1}
\end{figure}

Let $\eta_{j}$ and $\eta_{j+1}$ denote the number of $j$'s and $(j+1)$'s, respectively, in the RRW of $A_{(i \; i+1)}$ by the end of the $(i-1)^{st}$ row. Let $M = \eta_{j} - \eta_{j+1}$. Since $A$ is assumed to be Yamanouchi, we have that $M \geq b$ and $M \geq b+d-a$.

Let $x$ be the number of connected $j$'s on the left when executing the $R$-matrix algorithm. Similarly, let $y$ be the number of connected $(j+1)$'s on the left. Notice that $x \geq \min(a,d)$. Following the $R$-matrix algorithm, $j$ and $j+1$ occur in $B_{(i \; i+1)}$ as shown in Figure \ref{B2}.

\begin{figure}
\begin{center}
\includegraphics[scale=.25]{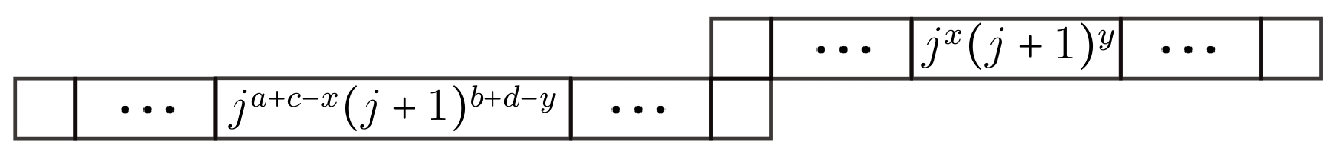}
\end{center}
\caption{Integers $j$ and $j+1$, as they appear in $B_{(i \; i+1)}$.}
\label{B2}
\end{figure}

Let $R$ be the word $(j+1)^y j^x (j+1)^{b+d-y} j^{a+c-x}$. Define the function $r(n)$ to be the number of $(j+1)$'s minus the number of $j$'s in the first $n$ elements of the word $R$. Clearly $r$ is maximal after a string of $(j+1)$'s, so either after the initial length $y$ string of $(j+1)$'s or after $b+d+x$ elements of $R$. We only have left to show that the function $r$ never exceeds $M$.

Notice that $r(y) = y$ and $r(b+d+x) = (y+(b+d-y)) - x = b + d -x $. Since $y \leq b \leq M$, we have that $r(y) \leq M$. For $r(b+d+x)$, we consider two cases. If $x \geq d$, then $r(b+d+x) = b + d -x \leq b \leq M$, as desired. On the other hand, if $x < d$, then since $x \geq \min (a,d)$ (as noted above), we have $x \geq a$. Then $r(b+d+x) = b + d -x \leq b + d - a \leq M$. This completes the proof.
\end{proof}

We have just shown that the $R$-matrix algorithm preserves the Yamanouchi property of ribbon LR-tableaux. Recall from the Preliminaries that the $R$-matrix operation as applied to a ribbon LR-tableau, preserves semistandardness within the two rows that are swapped; however, semistandardness of the entire ribbon is not necessarily preserved. Our next lemma will show that applying the $R$-matrix algorithm to rows $i$ and $i+1$ of a ribbon LR-tableau cannot increase the leftmost element in the $(i+1)^{st}$ row. This result is a step towards establishing how we might use the $R$-matrix algorithm while preserving the semistandardness of the entire ribbon.

\begin{lem} \label{R-bot-left}
Let $A$ be a ribbon LR-tableau of shape $\alpha = (\alpha_1, \ldots, \alpha_m)$ and let $i \in \{1,2,\ldots , m-1\}$. Denote by $A_{(i \; i+1)}$ the ribbon of shape $\a_{(i \; i+1)}$ obtained by applying the $R$-matrix algorithm to rows $i$ and $i+1$ of $A$. Let $x$ be the leftmost element of the $(i+1)^{st}$ row of $A$ and let $y$ be the leftmost element of the $(i+1)^{st}$ row of $A_{(i \; i+1)}$. Then $y \leq x$.
\end{lem}
\begin{proof}
Since $x$ is in the $(i+1)^{st}$ row of $A$, by the $R$-matrix algorithm, there is also an $x$ in the $(i+1)^{st}$ row of $A_{(i \; i+1)}$. The result now follows from the fact that the $R$-matrix operation preserves semistandardness within the two rows.
\end{proof}

The remaining way in which $A_{(i \; i+1)}$ may fail to be semistandard is for the number in the rightmost box of the $i^{th}$ row of $A_{(i \; i+1)}$ to be less than or equal to the number in the leftmost box of the $(i-1)^{st}$ row (where $A$ and $A_{(i \; i+1)}$ are as in Lemmas \ref{R-Yamanouchi} and \ref{R-bot-left}). This last obstacle is the main focus of the following proof.

\begin{thm} \label{Tri-Swap-Thing}
Let $\a = (\a_1, \a_2, \ldots, \a_m)$ be a ribbon, where each $\a_i \geq 2$ and $m \geq 3$. Suppose $\a_i > \a_{i+1}$ and $\a_i<\a_{i-1} + \a_{i+1}$ for some $1\leq i \leq (m-1)$, where $\a_0 = \infty$ for notational convenience. Then $[\a] \subseteq [\a_{(i \; i+1)}]$.
\end{thm}

\begin{proof}
Let $A$ be a ribbon LR-tableau with shape $\a$ and content $\mu$. By the Littlewood-Richardson rule, it suffices to show that there is some ribbon LR-tableau with shape $\a_{(i \; i+1)}$ and content $\mu$. 
 
Applying the $R$-matrix algorithm to rows $i$ and $i+1$ of $A$ yields a ribbon tableau $A_{(i \; i+1)}$ of shape $\a_{(i \; i+1)}$ and content $\mu$. By Lemma \ref{R-Yamanouchi}, $A_{(i \; i+1)}$ is Yamanouchi. If $A_{(i \; i+1)}$ is also semistandard, we are done. Recall that applying the $R$-matrix algorithm to rows $i$ and $i+1$ of $A$ preserves the semistandardness within the two rows. Moreover, we have by Lemma \ref{R-bot-left} that the leftmost entry in the $(i+1)^{st}$ row of $A_{(i \; i+1)}$ is not greater than that of $A$. Consequently, if $i = 1$, then $A_{(i \; i+1)}$ is semistandard, and we are done. Supposing $i>1$, the only way in which $A_{(i \; i+1)}$ can fail to be semistandard is if the rightmost entry in the $i^{th}$ row of $A_{(i \; i+1)}$ is less than or equal to the leftmost entry in the $(i-1)^{st}$ row of $A_{(i \; i+1)}$. Assume that this is the case (i.e. that $A_{(i \; i+1)}$ is not semistandard). In the remainder of the proof, we show that $A_{(i \; i+1)}$ can be modified to produce a tableau which is both semistandard and Yamanouchi. Each modification will preserve both shape and content, so the resulting tableau will still have shape $\a_{(i \; i+1)}$ and content $\mu$.

\begin{figure}
\begin{center}
\includegraphics[scale=.25]{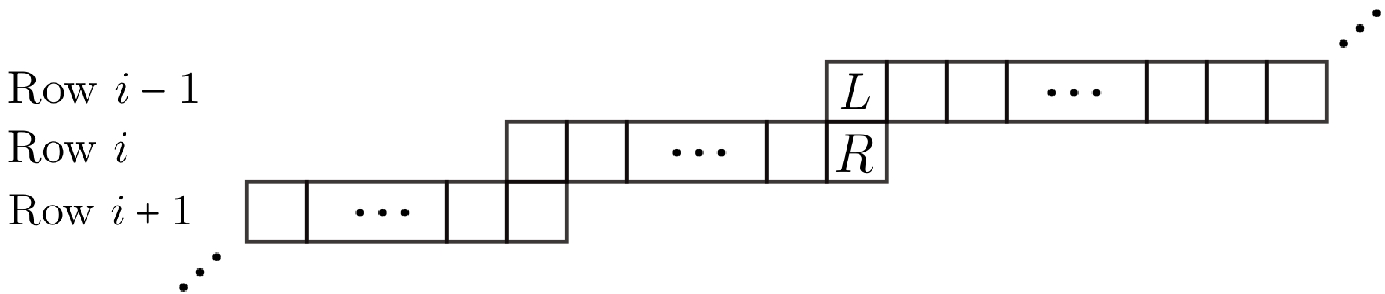}
\includegraphics[scale=.25]{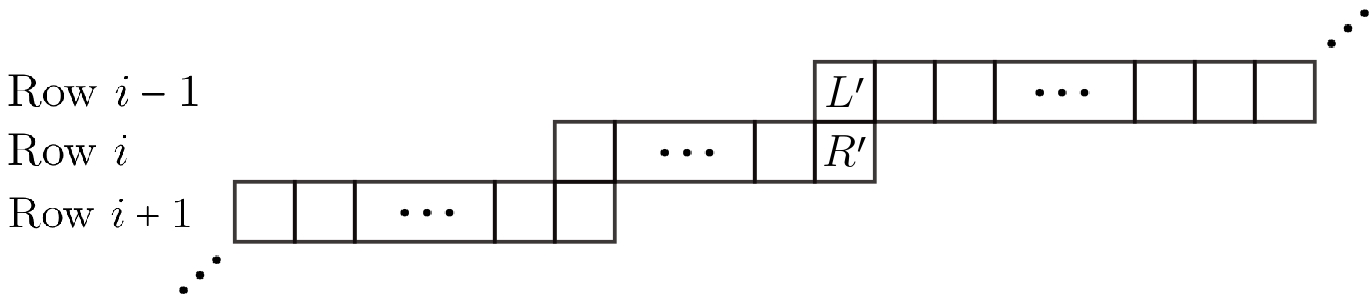}
\end{center}
\caption{Box labellings of $A$ (top) and $A_{(i \; i+1)}$ (bottom).}
\label{alpha-boxes}
\end{figure}

Let $L$ (resp. $L'$) denote the leftmost box in the $(i-1)^{st}$ row of $A$ (resp. $A_{(i \; i+1)}$). Similarly, let $R$ (resp. $R'$) denote the rightmost box in the $i^{th}$ row of $A$ (resp. $A_{(i \; i+1)}$). These labellings are shown in Figure \ref{alpha-boxes}. Let $\ell$, $r$, $\ell'$, and $r'$ be the entries in boxes $L$, $R$, $L'$, and $R'$, respectively. By our assumption that $A_{(i \; i+1)}$ is not semistandard, we have that $\ell' \geq r'$. If $\ell' > r'$, we can simply swap the positions of $\ell'$ and $r'$ to obtain a semistandard tableau. This swap preserves the Yamanouchi property, as the RRW of the resulting tableau can be obtained from that of $A_{(i \; i+1)}$ without moving $i$ to be before $i-1$ for any $2 \leq i \leq m$. We can therefore assume that $\ell' = r'$. 

Let $w$ denote the leftmost entry in the $(i-1)^{st}$ row of $A_{(i \; i+1)}$ that is greater than $r' = \ell'$ (if such a $w$ exists). If $w$ is not the rightmost entry of the row, then we can swap it with the $r'$ in box $R'$, yielding a semistandard tableau $A_{(i \; i+1),w}$. That $A_{(i \; i+1),w}$ is Yamanouchi follows easily from the fact that $A_{(i \; i+1)}$ is Yamanouchi. 

Therefore, in the case that $w$ exists and is not the rightmost entry of row $i-1$, we have obtained an LR-tableau with content $\mu$. Thus we may assume that all entries, except possibly the rightmost, of the $(i-1)^{st}$ row of $A_{(i \; i+1)}$ equal $r'=\ell'$. Recall that $A$ and $A_{(i \; i+1)}$ differ only in their $i^{th}$ and $(i+1)^{st}$ rows. Therefore, our assumptions regarding the entries of the $(i-1)^{st}$ row of $A_{(i \; i+1)}$ apply also to the $(i-1)^{st}$ row of $A$. These assumptions are depicted in Figure \ref{fig:Update1}.

\begin{figure}
\begin{center}
\includegraphics[scale=.25]{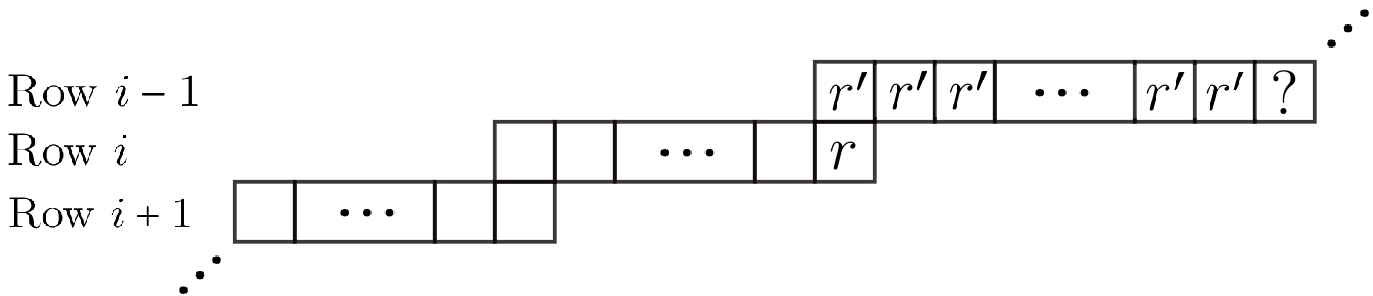}
\includegraphics[scale=.25]{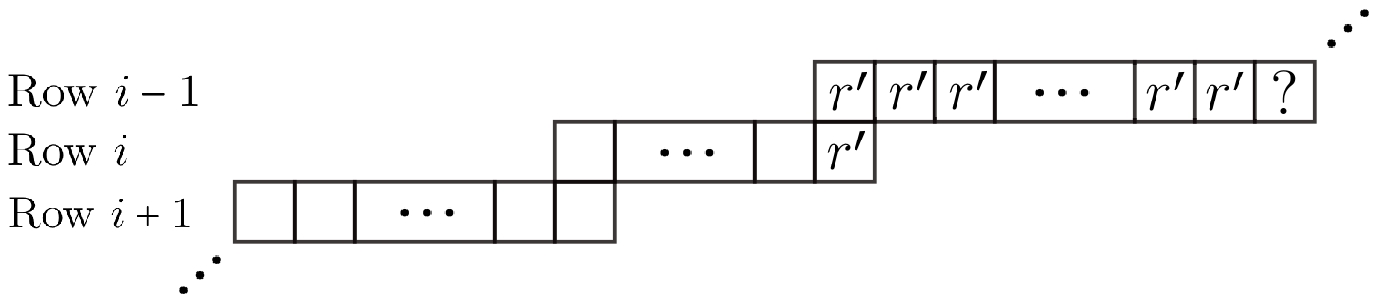}
\end{center}
\caption{Assumed entries of $A$ (top) and $A_{(i \; i+1)}$ (bottom).}
\label{fig:Update1}
\end{figure}

Similarly, let $x$ be the rightmost entry of the $i^{th}$ row of $A_{(i \; i+1)}$ that is less than $r'$ (if such an $x$ exists). Unless $x$ is the leftmost entry of the row, we can swap $x$ with the $r'$ in box $L'$ to obtain a semistandard tableau. As before, the resulting tableau is Yamanouchi. We can therefore assume that all entries in the $i^{th}$ row of $A_{(i \; i+1)}$, except possibly the leftmost, are equal to $r'$.

Since $A$ is semistandard, the element $r$ in box $R$ of $A$ must be strictly greater than $r'$. Recall that the $i^{th}$ and $(i+1)^{st}$ rows of $A$ and $A_{(i \; i+1)}$ have the same content. Noting that all entries in the $i^{th}$ row of $A_{(i \; i+1)}$ are at most $r'$, there must therefore be an $r$ in the $(i+1)^{st}$ row of $A_{(i \; i+1)}$.

Now, let $y$ denote the leftmost entry of the $i^{th}$ row of $A_{(i \; i+1)}$, and consider the case where $y \neq r'$. Using that the $R$-matrix algorithm preserves semistandardness within the rows it swaps, we have that $y < r'$. We can therefore swap this $y$ with the $r'$ in box $L'$ of $A_{(i \; i+1)}$ to obtain a semistandard tableau. As before, this swap does not violate the Yamanouchi property. We may therefore assume that $y = r'$. By the above arguments, we may now assume that the entries of $A_{(i \; i+1)}$ are as shown in Figure \ref{fig:alpha-i-i+1-fill}.

\begin{figure} 
\begin{center}
\includegraphics[scale=.25]{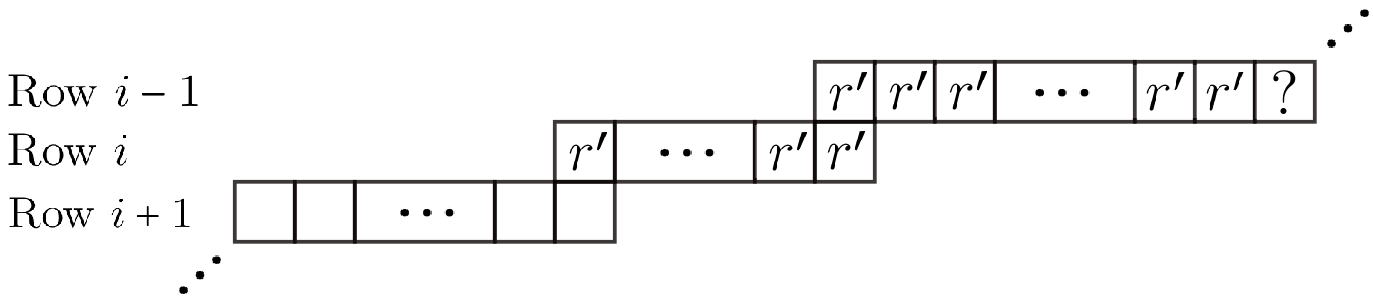}
\end{center}
\caption{Assumed entries of $A_{(i \; i+1)}$.}
\label{fig:alpha-i-i+1-fill}
\end{figure}

From here, the idea is to swap the rightmost $r'$ in the $i^{th}$ row of $A_{(i \; i+1)}$ with the appropriate entry from the row below to obtain an LR-tableau. We will show, using our triangle inequality assumption, that such a swap can be made in all but a very narrow set of cases. Finally, we will handle the exceptional cases by scanning $A_{(i \; i+1)}$ upwards looking for potential swaps.

Consider the leftmost box in the $(i+1)^{st}$ row of $A_{(i \; i+1)}$ with entry greater than $r'$; call this entry $z$. We claim that the aforementioned box cannot be the rightmost box of the $(i+1)^{st}$ row of $A_{(i \; i+1)}$. Suppose, for the sake of contradiction, that this box is in fact the rightmost box of the $(i+1)^{st}$ row of $A_{(i \; i+1)}$. Then each of $A$ and $A_{(i \; i+1)}$ have only one entry that is greater than $r'$ (namely, $z$) in their $i^{th}$ and $(i+1)^{st}$ rows. In particular, using our previous notation (see Figure \ref{fig:Update1}), we have that $r = z$. Since there is only one element greater than $r'$ in these two rows, we have by the semistandardness of $A$ that the leftmost element in the $i^{th}$ row of $A$ is strictly less than $r'$ (otherwise two $r'$ would be in the same column of $A$). However, the $R$-matrix algorithm makes it impossible for both $z$ and the entry less than $r'$ to be moved to the $(i+1)^{st}$ row, while $r'$'s remain in the $i^{th}$ row. We therefore conclude that the leftmost box in the $(i+1)^{st}$ row of $A_{(i \; i+1)}$ with entry greater than $r'$ is not the rightmost box of the row. Consequently, swapping the leftmost $z$ of the $(i+1)^{st}$ row of $A_{(i \; i+1)}$ with the $r'$ in box $R'$ of $A_{(i \; i+1)}$ will produce a semistandard tableau $A_{(i \; i+1)}^{z}$ (see Figure \ref{fig:z-tab}). If $A_{(i \; i+1)}^{z}$ is Yamanouchi, then the proof is complete. We now show that $A_{(i \; i+1)}^{z}$ can only fail to be Yamanouchi in a very narrow set of cases.

\begin{figure} 
\begin{center}
\includegraphics[scale=.25]{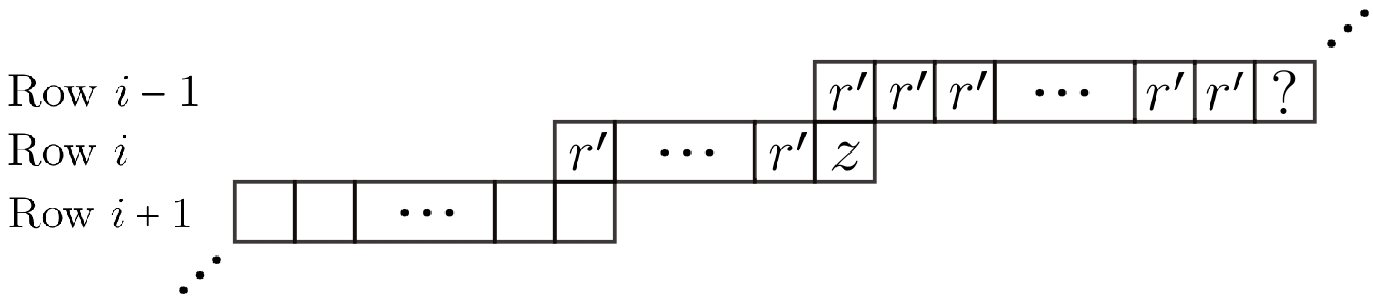}
\end{center}
\caption{Entries of $A_{(i \; i+1)}^{z}$.}
\label{fig:z-tab}
\end{figure}

Let $X$ denote the subword of the RRW of $A_{(i \; i+1)}$ and $A_{(i \; i+1)}^z$ formed by the entries strictly between the entries we swapped to obtain $A_{(i \; i+1)}^{z}$ from $A_{(i \; i+1)}$ (see Figure \ref{fig:xwz}). If $z \neq r'+1$, then this subword does not contain either $r' + 1$ or $z -1$. Consequently, this swap preserves the Yamanouchi property when $z \neq r'+1$. We may therefore assume that $z = r' +1$. Thus, given that $A_{(i \; i+1)}$ is Yamanouchi, the only way in which $A_{(i \; i+1)}^{z}$ can fail to be Yamanouchi is if the following phenomenon occurs: within the subword $X$, the number of $z$'s in the RRW of $A_{(i \; i+1)}^{z}$ overtakes the number of $r'$'s in the RRW of $A_{(i \; i+1)}^{z}$. To understand the narrow set of cases in which this phenomenon can occur, let us further investigate the RRW of $A_{(i \; i+1)}^{z}$.

Noting that the subword $X$ begins with $\a_{i+1}-1$ occurrences of $r'$, it suffices to investigate the portion of $X$ determined by elements in the $(i+1)^{st}$ row of $A_{(i \; i+1)}^{z}$. If there are no $z$'s in this row, then there is nothing to check. Therefore, assume that there is at least one $z$ in the $(i+1)^{st}$ row of $A_{(i \; i+1)}^{z}$. By the semistandardness of $A_{(i \; i+1)}^{z}$, this means that there is exactly one contiguous string of $z$'s in this portion of $X$ (say, of length $k$). Moreover, by definition of $X$, this contiguous string is a suffix of $X$.

Let $Z$ denote the prefix of the RRW of $A_{(i \; i+1)}^{z}$ formed by truncating after $X$. It follows that if $A_{(i \; i+1)}^{z}$ violates the Yamanouchi property, then so does $Z$. Also, let $W$ denote the prefix of the RRW of $A_{(i \; i+1)}$ and $A_{(i \; i+1)}^{z}$ of length $\sum_{j=1}^{i-2} (\a_j) +1$ (i.e. the prefix ending immediately after the rightmost element of the $(i-1)^{st}$ row). Note that we can write $Z = W \; (r')^{\a_{i-1}-1} \; z \; X$ (see Figure \ref{fig:xwz}). The prefix of $A_{(i \; i+1)}$ of the same length can be written as $W \; (r')^{\a_{i-1}} \; X$ (see Figure \ref{fig:xwz}).

\begin{figure} 
\begin{center}
\includegraphics[scale=.21]{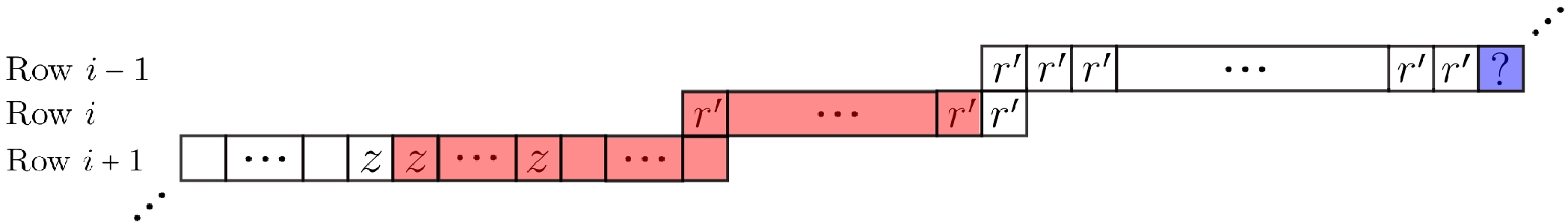}
\includegraphics[scale=.21]{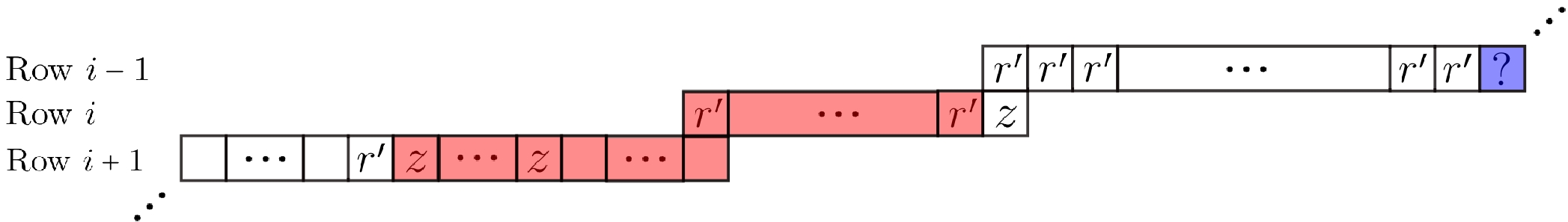}
\end{center}
\caption{This figure (top) depicts entries of $A_{(i \; i+1)}$ as well as subwords of the RRW of $A_{(i \; i+1)}$. The red shading indicates entries corresponding to the subword $X$, while the blue shading indicates the last element of the subword $W$. This figure (bottom) also depicts the analogous entries and subwords of $A_{(i \; i+1)}^{z}$.}
\label{fig:xwz}
\end{figure}

Let $d$ be the number of $r'$'s that occur in $W$ minus the number of $z$'s that occur in $W$. Since $W$ is Yamanouchi and $z = r'+1$, we have that $d \geq 0$. Recalling that $A_{(i \; i+1)}$ is Yamanouchi, we have
\begin{equation} \label{eq:thing1}
d + (\a_{i-1} -1) + \a_{i+1} -(k+1) \geq 0.
\end{equation}
Suppose that $A_{(i \; i+1)}^{z}$ is not Yamanouchi. Then
\begin{equation} \label{eq:thing2}
d + (\a_{i-1}-1) -1 + (\a_{i+1}-1) - k = (d + (\a_{i-1} -1) + \a_{i+1} -(k+1)) - 1 < 0.
\end{equation}
Together, (\ref{eq:thing1}) and (\ref{eq:thing2}) give that $d + (\a_{i-1} -1) + \a_{i+1} -(k+1) = 0$. Noting that $k+1 \leq \a_i \leq \a_{i-1} + \a_{i+1}-1$, we have
\begin{align*}
d &= d + (\a_{i-1} + \a_{i+1} -1) - (\a_{i-1} + \a_{i+1} -1) \\
& \leq d+ \a_{i-1} + \a_{i+1} -1 - \a_i \\
& \leq d + (\a_{i-1}-1) + \a_{i+1} - (k+1) = 0.
\end{align*}
Recalling that $d \geq 0$, this gives that $d = 0$. This is the crucial consequence of the rows satisfying a strict triangle inequality. 

Plugging this into (\ref{eq:thing2}) gives that $\a_{i-1} + \a_{i+1} -1 \leq k+1$. Thus, we have
$$\a_{i-1} + \a_{i+1} -1 \leq k+1 \leq \a_i \leq \a_{i-1} + \a_{i+1} - 1,$$
meaning $\a_{i-1} + \a_{i+1} -1 = k+1 = \a_i$. This means that all $\a_i$ entries in the $(i+1)^{st}$ row of $A_{(i \; i+1)}$ are $z$'s.

Our consideration of $A_{(i \; i+1)}^{z}$ is now complete; the assumption that $A_{(i \; i+1)}^{z}$ is not Yamanouchi has allowed us to derive powerful constraints on $A_{(i \; i+1)}$. We will now complete the proof by investigating $A_{(i \; i+1)}$. In particular, we will scan $A_{(i \; i+1)}$ upwards from the $i^{th}$ row, arguing that we must eventually find an entry we can swap with the $r'$ in either box $L'$ or box $R'$ of $A_{(i \; i+1)}$ to obtain an LR-tableau.

\begin{figure}
\begin{center}
\includegraphics[scale=.235]{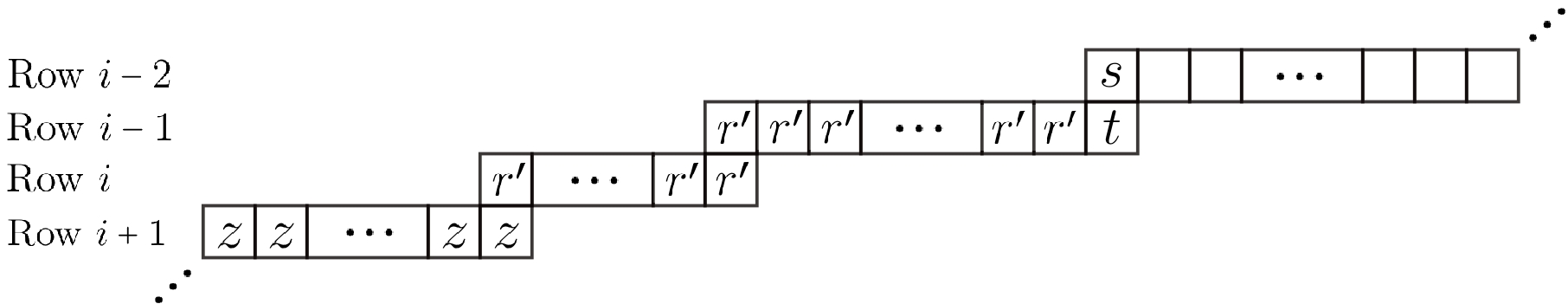}
\end{center}
\caption{Assumed entries of $A_{(i \; i+1)}$.}
\label{fig:FinalTab}
\end{figure}

Let $s$ denote the leftmost entry in the $(i-2)^{nd}$ row of $A_{(i \; i+1)}$, and let $t$ denote the rightmost entry in the $(i-1)^{st}$ row of $A_{(i \; i+1)}$ (see Figure \ref{fig:FinalTab}). (Note that the $(i-1)^{st}$ row cannot be the top row of $A_{(i \; i+1)}$, as this would require that $t = r' = 1$, which would be incompatible with the assumptions that $d = 0$ and that $A_{(i \; i+1)}$ is Yamanouchi.)

We will now argue that we can assume that $s=r'$. First note that $t \neq r'$, as if $t=r'$, then the assumption that $d=0$ would imply that the prefix of the RRW ending immediately before $t$ contains one more $z$ than $r'$, a violation of the Yamanouchi property. Therefore $t > r'$. 

If $s < r'$, then the $r'$ in box $R'$ can be swapped with $t$ to obtain an LR-tableau. If $s > r'$, then we can perform the following two swaps to obtain an LR-tableau: swap the $r'$ in box $R'$ with $t$, and then swap the same $r'$ with the $s$ above it. We can therefore assume that $s=r'$. 

Moreover, we claim that $t = z = r' +1$. This is because $d = 0$; if $t \neq z$, then the prefix of the RRW of $A_{(i \; i+1)}$ formed by truncating immediately before $s$ is not Yamanouchi. These updated assumptions are shown in Figure \ref{fig:FinalUpdated}. 

\begin{figure}
\begin{center}
\includegraphics[scale=.235]{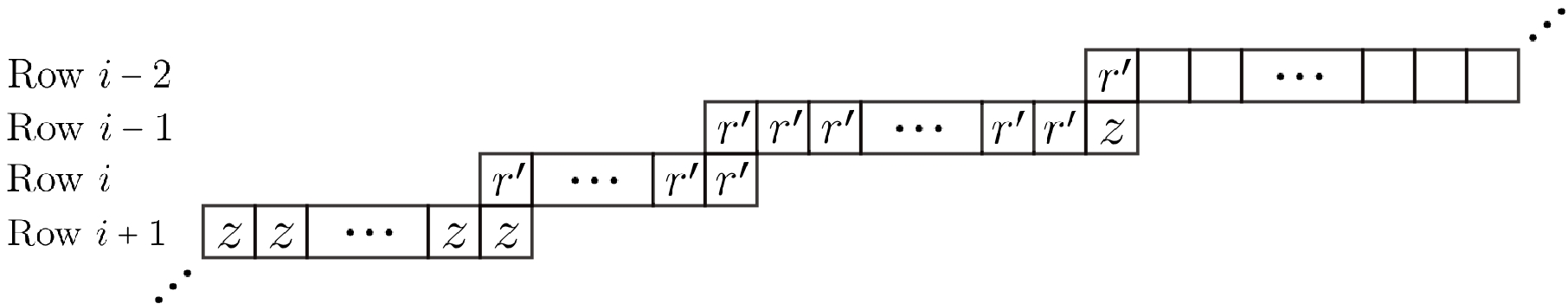}
\end{center}
\caption{Assumed entries of $A_{(i \; i+1)}$.}
\label{fig:FinalUpdated}
\end{figure}

A row of $A_{(i \; i+1)}$ will be called \textit{trivial} if it is of length two and its entries are ``$r' \; z$'' when reading from left to right across the row; otherwise, the row will be called \textit{nontrivial}. There must be a nontrivial row among the first $i-2$ rows of $A_{(i \; i+1)}$, as otherwise the first letter of the RRW of $A_{(i \; i+1)}$ would be $z > r'\geq 1$. Choose the maximal $j \leq i-2$ such that the $j^{th}$ row of $A_{(i \; i+1)}$ is nontrivial.

By our choice of $j$, all rows strictly between the $j^{th}$ and $(i-1)^{st}$ rows are trivial, so the rightmost box of the $(j+1)^{st}$ row must contain a $z$ (even if $j = i-2$). Therefore, we have by the semistandardness of $A_{(i \; i+1)}$ that the leftmost entry $u$ of the $j^{th}$ row is at most $r' = z-1$.

\begin{figure}
\begin{center}
\makebox[\textwidth][c]{\includegraphics[width=1.2\textwidth]{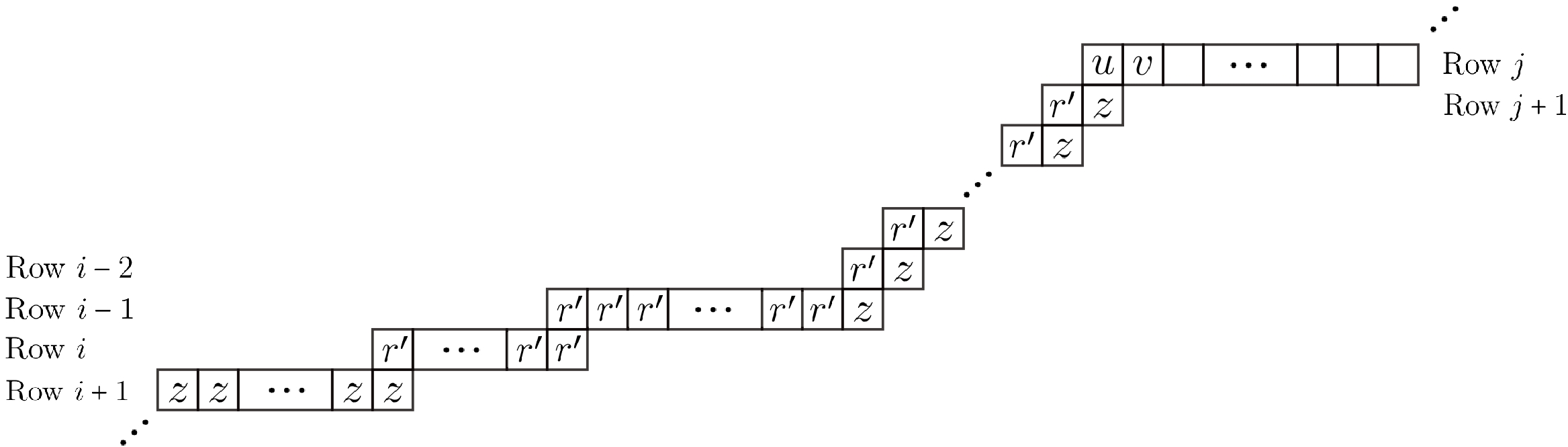}}
\makebox[\textwidth][c]{\includegraphics[width=1.2\textwidth]{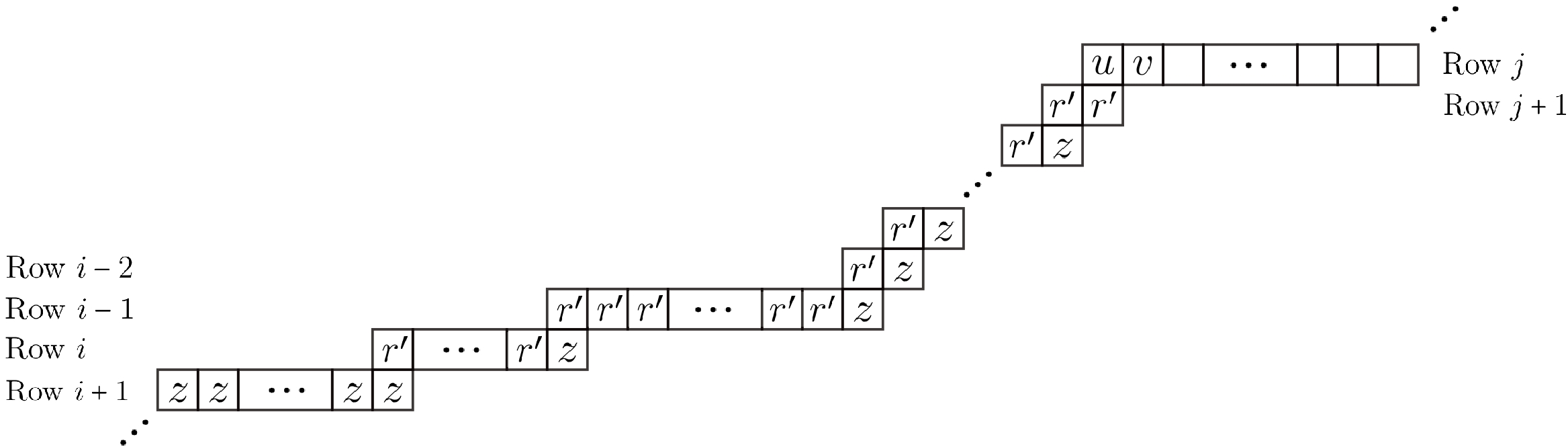}}
\end{center}
\caption{Assumed entries of $A_{(i \; i+1)}$ (top) and $A_{(i \; i+1),z}$ (bottom).}
\label{fig:FinalUpdate2}
\end{figure}

If $u < r'$, then the $z$ in the rightmost box of the $(j+1)^{st}$ row can be swapped with the $r'$ in box $R'$ of $A_{(i \; i+1)}$ to obtain a semistandard tableau $A_{(i \; i+1), z}$ (see Figure \ref{fig:FinalUpdate2}). Moreover, $A_{(i \; i+1), z}$ is Yamanouchi, since to obtain the RRW of $A_{(i \; i+1), z}$ from that of $A_{(i \; i+1)}$, one does not move $r'$ in front of any $(r'-1)$'s, nor does one move $z$ behind any $(z+1)'s$. We can therefore assume that $u = r'$.

Let $v$ denote the entry to the right of $u$ in the $j^{th}$ row of $A_{(i \; i+1)}$ (see Figure \ref{fig:FinalUpdate2}). Recall that $d = 0$ and that all rows strictly between the $j^{th}$ and $(i-1)^{st}$ rows are trivial. As a result, if $v = r'$ as well, then the prefix of the RRW ending immediately before $v$ is not Yamanouchi. We can therefore assume that $v > r'$.

Suppose that $\a_j > 2$. Then $v$ is not the rightmost element of the $j^{th}$ row of $A_{(i \; i+1)}$, so we can swap it with the $r'$ in box $R'$ of $A_{(i \; i+1)}$ to obtain a semistandard tableau $A_{(i \; i+1),v}$. Moreover, $A_{(i \; i+1),v}$ is Yamanouchi, since to obtain its RRW from that of $A_{(i \; i+1)}$, one does not move $v$ behind any instances of $v+1$, nor does one move $r'$ ahead of any instances of $r'-1$. Therefore, let us assume that $\a_j = 2$.

If the leftmost entry in the $(j-1)^{st}$ row of $A_{(i \; i+1),v}$ (call it $q$) is less than $r'$, then $A_{(i \; i+1),v}$ is again an LR-tableau. Therefore, assume that $q \geq r'$, in which case $A_{(i \; i+1), v}$ is not semistandard. If $q > r'$, we can swap $q$ with the $r'$ beneath it to obtain an LR-tableau. We can therefore assume that $q = r'$.

We now return to considering the tableau $A_{(i \; i+1)}$, with the assumptions that $\a_j = 2$ and $q = r'$. Recall again that $d = 0$ and that all rows strictly between the $j^{th}$ and $(i-1)^{st}$ rows of $A_{(i \; i+1)}$ are trivial. If $v \neq z$, then the prefix of the RRW of $A_{(i \; i+1)}$ ending immediately before $q$ is not Yamanouchi. Consequently, we have that $v = z$, contradicting our assumption that the $j^{th}$ row of $A_{(i \; i+1)}$ is nontrivial. This completes the proof.
\end{proof}

Since transpositions generate the symmetric group, Theorem \ref{Tri-Swap-Thing} allows us to prove the following sufficient condition for a ribbon to have full equivalence class. This result generalizes the finding of McNamara and van Willigenburg that all equitable ribbons have full equivalence class \cite[Thm. 1.5]{McWilligenburgSupport}.

\begin{cor} \label{Gen-Tri-Thing}
Let $\a = (\a_1, \a_2, \ldots, \a_m)$ be a ribbon with each $\a_i \geq 2$ and $m \geq 3$. Set $\a_0 = \infty$ for convenience of notation. If all 3-subsets of $\{\a_i\}_{i=1}^m$ satisfy the strict triangle inequality, then $\a$ has full equivalence class.
\end{cor}

\begin{proof}
Let $i \in \{1,2,\ldots, m-1 \}$ be arbitrary. As noted above, it suffices to show that $[\a] = [\a_{(i \; i+1)}]$. If $\a_i = \a_{i+1}$, then $[\a] = [\a_{(i \; i+1)}]$ follows trivially. Thus, by Remark \ref{rmk:antipodal}, we can assume without loss of generality that $\a_i > \a_{i+1}$. By assumption, $\a_i < \a_{i-1} + \a_{i+1}$, so Theorem \ref{Tri-Swap-Thing} implies that $[\a] \subseteq [\a_{(i \; i+1)}]$.

To show containment the other way, we consider the antipodal rotation $\a^{\circ} = (\a^\circ_1, \a^\circ_2, \ldots \a^\circ_m) = (\a_m, \a_{m-1}, \allowbreak \ldots , \a_1)$ of $\a$, with $\a^{\circ}_0 = \infty$ for convenience. Set $j = m+1 - i$, so that $\a^{\circ}_j = \a_i$ and $\a^{\circ}_{j-1} = \a_{i+1}$. Note that $2 \leq j \leq m$ since $1 \leq i \leq m-1$. Since $(\a_{(i, \; i+1)})^\circ = \a^\circ_{(j-1 \; j)}$, we have that $[\a_{(i \; i+1)}] = [\a^\circ_{(j-1\; j)}]$. Hence, it will suffice to show that $[\a^{\circ}_{(j-1 \; j)}] \subseteq [\a^{\circ}]$.

Since $\a_i > \a_{i+1}$, we have that $\a^{\circ}_j > \a^{\circ}_{j-1}$ (i.e. the $(j-1)^{st}$ row of $[\a^{\circ}_{(j-1 \; j)}]$ is longer than the $j^{th}$ row of $[\a^{\circ}_{(j-1 \; j)}]$). Moreover, we have by assumption that $\a^{\circ}_j < \a^{\circ}_{j-1} + \a^{\circ}_{j-2}$ (meaning that the $(j-1)^{st}$ row of $[\a^{\circ}_{(j-1 \; j)}]$ is shorter than the combined lengths of the row immediately above it and the row immediately below it). Thus, we have established that the ribbon $[\a^{\circ}_{(j-1 \; j)}]$ meets the conditions from Theorem \ref{Tri-Swap-Thing} with respect to its $(j-1)^{st}$ and $j^{th}$ rows.

Since swapping $\a^{\circ}_j$ and $\a^{\circ}_{j-1}$ in $\a^{\circ}_{(j-1 \; j)}$ gives us back $\a^{\circ}$, Theorem \ref{Tri-Swap-Thing} implies that $[\a^{\circ}_{(j-1 \; j)}] \subseteq [\a^{\circ}]$, as desired.
\end{proof}

Having proven a sufficient condition for a ribbon to have full equivalence class, we now turn to our separate necessary condition. 

\section{A Necessary Condition} \label{sec:necessary}

In this section, we will prove the following necessary condition for a ribbon to have full equivalence class.

\begin{thm} \label{strong-nec}
Let $\alpha = (\a_1, \a_2, \ldots, \a_m)$ be a ribbon with $\a_1 \geq \a_2 \geq \cdots \geq \a_m$, where each $\a_i \geq 2$, and $m \geq 3$. If $[\alpha] =  [\a_{(j \; j+1)}]$ , then $N_j < \sum_{i=j+1}^m \alpha_i - (m-j-2)$, where 
$$N_j := \max \Bigg\{k : \sum_{i \leq j: \; \alpha_i < k} (k-\alpha_i) \leq m-j-2 \Bigg\}. $$
In particular, if $\a$ has full equivalence class, then,  for all $1 \leq j \leq m-2$, we have that $N_j < \sum_{i=j+1}^m \alpha_i - (m-j-2)$.
\end{thm}

This condition is certainly less penetrable than our sufficient condition, so before delving into the proof, we will illustrate the necessary condition with an example and a non-example. Additionally, the proof we will give is a constructive one, which we hope will make the above definition of $N_j$ more transparent.

\begin{example} \label{ex: necessary condition}
Let $\a = (10,8,6,5,4)$ be a ribbon. Then $\a$ does not satisfy our sufficient condition of Corollary \ref{suff-real} (as $10 \geq 5+4$). However, we now show that it does satisfy the necessary condition of Theorem \ref{strong-nec}. Since $m=5$ in this example, checking the necessary condition amounts to checking the inequalities corresponding to $j=1,2,3$:
\begin{itemize}
\item We have 
$$N_1 = \max \left \{k \; : \sum_{i \leq 1: \; \alpha_i < k} (k-\alpha_i) \leq 2 \right \} = \max \left \{k \; : (k-10) \leq 2 \right \} = 12.$$
Since $\displaystyle 12< \sum_{i=2}^5 \a_i - 2 = 21$, the necessary inequality holds for $j=1$.
\item We have 
$$N_2 = \max \left \{k \; : \sum_{i \leq 2: \; \alpha_i < k} (k-\alpha_i) \leq 1 \right \} = 9,$$
since $9-8 \leq 1$, but $10-8 > 1$. Additionally, $\displaystyle 9 < \sum_{i=3}^5 \a_i - 1 = 14$, the inequality corresponding to $j=2$ holds.
\item We have
$$N_3 = \max \left \{k \; : \sum_{i \leq 3: \; \alpha_i < k} (k-\alpha_i) \leq 0 \right \} =6,$$
since the summation when $k=6$ is the empty sum (and hence is zero), whereas $\displaystyle \sum_{i \leq 3: \; \alpha_i < 7} (7-\alpha_i) = 7-6 = 1 > 0$. Since $\displaystyle 6 < \sum_{i=4}^5 \a_i -0 = 9$, the inequality corresponding to $j=3$ holds.
\end{itemize}
\end{example}

\begin{non-example} \label{non-ex: necessary condition}
Let $\a = (13, 10, 5, 4, 3)$. Then 
$$N_2 = \max \left \{k \; : \sum_{i \leq 2: \; \alpha_i < k} (k-\alpha_i) \leq 1 \right \} = 11.$$ 
Since $\displaystyle 11 \geq \sum_{i=3}^5 \a_i - 1 = 11$, this shows that the necessary inequality does not hold for $j=2$. As a result, Theorem \ref{strong-nec} tells us that $\a$ does not have full equivalence class. 
\end{non-example}

\begin{remark} \label{rem:nec}
Using the notation from Theorem \ref{strong-nec}, we will always have $\alpha_j \leq N_j \leq \alpha_j + m - j - 2$. In particular, $N_j = \alpha_j$ whenever $j = m-2$, while $N_j = \alpha_j + m - j - 2$ if and only if $\alpha_j \leq \alpha_{j-1} - (m-j-2)$.
\end{remark}
\begin{remark} \label{rem:alg}
In general, to efficiently determine $\displaystyle N_j = \max \left \{k : \sum_{i \leq j: \; \alpha_i < k} (k-\alpha_i) \leq m-j-2 \right \}$, one can use the following algorithm:
\begin{enumerate}
\item Start with $k=\a_j$.
\item Set $s$ = 0. For each $\a_i \in \{\a_1, \a_2, ..., \a_j\}$ with $k > \a_i$, add $(k-\a_i)$ to $s$.
\item If $s > m-j-2$, then $N_j = k-1$. Otherwise, put $k=k+1$ and go back to step 2.
\end{enumerate}
\end{remark}

\begin{remark}
A much weaker but simpler version of our necessary condition is that $\alpha_i < \sum_{k=i+1}^m \alpha_k$ for all $1 \leq i \leq m-2$.
\end{remark}

Having gained some familiarity with the necessary condition, we now turn towards its proof, which we begin with a lemma.

\begin{lem} \label{j-longer-than-j+1}
Let $\alpha = (\a_1, \a_2, \ldots, \a_m)$ be a ribbon with $\a_1 \geq \a_2 \geq \cdots \geq \a_m$, where each $\a_i \geq 2$, and $m \geq 3$. If $\displaystyle N_j \geq \sum_{i=j+1}^m \alpha_i - (m-j-2)$ for some $j \in \{1,2,\ldots,m-2 \}$, then $\a_j > \a_{j+1}$.
\end{lem}

\begin{proof}
The inequality  $\displaystyle N_j \geq \sum_{i=j+1}^m \alpha_i - (m-j-2)$ implies that
$$
\a_{j+1}  \leq N_j + (m-j-2) - \sum_{i=j+2}^{m} \a_i 
 \leq \a_j + 2(m-j-2) - \sum_{i=j+2}^{m} \a_i 
< \a_j,
$$
where the second inequality comes from the upper bound on $N_j$ given in Remark \ref{rem:nec}, and the third inequality follows from the assumption that all rows are at least two boxes long.
\end{proof}

\begin{proof}[Proof of Theorem~\ref{strong-nec}]
We prove the contrapositive. Fix $j \in \{1,2,\ldots , m-2\}$ such that $\displaystyle N_j \geq \sum_{i=j+1}^m \alpha_i - (m-j-2)$. We will exhibit a ribbon LR-tableau of shape $\alpha_{(j \; j+1)}$, whose content we will call $\mu$; then we will prove that there is no ribbon LR-tableau of shape $\a$ and content $\mu$.

Fill $\a_{(j \; j+1)}$ as follows (we'll call the resulting tableau $A$). Fill the $i^{th}$ row entirely with $i$'s for $i \leq j$. Put $\alpha_{j+1}$ $(j+1)$'s in the rightmost boxes of the $(j+1)^{st}$ row and fill the remaining boxes in this row with $j$'s. Note that by Lemma \ref{j-longer-than-j+1}, the leftmost entry of the $(j+1)^{st}$ row in this filling is a $j$ (meaning this row is longer than $\alpha_{j+1}$ in length).

We now fill the remaining $m-j-1$ rows with as many $(j+1)$'s as possible (while maintaining semistandardness, but perhaps not the Yamanouchi property, although we will show with the upcoming arguments that it indeed is Yamanouchi); put $(j+1)$'s in all but the leftmost box of the next $m-j-2$ rows, as well as in every box in the last row. Now the only empty boxes are the leftmost boxes in rows $j+2, j+3 \ldots, m-2, m-1$. We will call these remaining boxes \textit{critical boxes}. Fill the critical boxes from top to bottom according to the following algorithm: in each box, put the largest integer $\leq j$ such that the prefix of RRW through that box remains Yamanouchi. In practice, this means we will use exclusively $j$'s until the number of $j$'s in the tableau equals the number of $(j-1)$'s. Then, we will alternate between $(j-1)$'s and $j$'s until the both the number of $j$'s and of $(j-1)$'s equals the number of $(j-2)$'s. At this point, we rotate between placing $j$'s, $(j-1)$'s, and $(j-2)$'s until the number of each of these equals the number of $(j-3)$'s. We continue in this manner until all boxes have been filled. Towards proving that the resulting tableau $A$ is Yamanouchi, we first show that it contains exactly $N_j$ $j$'s. 

First we define a \textit{round}. Let the variables $c_1, c_2, \ldots, \allowbreak c_{m-j-2}$ represent the entries in the critical boxes, from top to bottom. Let $J = \{c_s: c_s = j\}$. Now partition $\{c_1 , \dots , c_{m-j-2}\}$ into \textit{rounds}, where each round is a consecutive subsequence of $c_1,\ldots, c_{m-j-2}$ whose last element is in $J$ but with no other elements in $J$ (i.e. a round ends if and only if an element equal to $j$ is encountered). \\

\noindent \emph{Claim}: If $r$ rounds can be completed before reaching the bottom, then at the end of the $r^{th}$ round, we have filled exactly $\displaystyle \sum_{i \leq j: \; \alpha_i < \a_j + r} (\a_j + r-\alpha_i)$ critical boxes. In particular, after the $r^{th}$ round, each number $i \leq j$ such that $\a_i \leq \a_j + r$ has occurred in exactly $\a_j + r - \a_i$ critical boxes.

\begin{proof} [Proof of Claim]
We will use induction on $r$. The claim trivially holds when $r=0$. Now consider an arbitrary $r>0$ (such that $r$ rounds can be completed before reaching the bottom) and assume the claim holds for $r-1$. In the $r^{th}$ round, we will write every number that was used in the $(r-1)^{st}$ round one more time, as well as any number $\ell$ satisfying $\a_\ell = \a_j + r-1$. Therefore, the latter numbers $\ell$ will each fill exactly one critical box after $r$ rounds, as is appropriate since, by choice of $\ell$, $\a_j + r - \a_\ell = 1$. All numbers which appeared in the $(r-1)^{st}$ round have now occurred in a critical box one more time than before. For a fixed number $i$, by the induction hypothesis, this is $\a_j + (r-1) - \a_i + 1 = \a_j + r - \a_i$ times. This completes the proof of the claim.
\end{proof}

Clearly the number of $j$'s in $A$ is $\a_j$ plus the number of rounds executed before running out of critical boxes. That is, if $\mu_j$ is the number of $j$'s in $A$, 
\begin{align*}
\mu_j &= \a_j + \max \Bigg \{r \; : \sum_{i\leq j: \a_i < \a_j + r} \a_j + r - \a_i \leq m-j-2 \Bigg \}\\
&= \max \Bigg\{k: \sum_{i \leq j: \a_i < k} (k-\a_i) \leq m-j-2 \Bigg \} = N_j.
\end{align*}
In particular, $\mu_j = N_j$.

By construction, the number of $(i+1)$'s never surpasses the number of $i$'s in the RRW of $A$, since $1 \leq i \leq j-1$. Semistandardness is also clear by construction, so all that is left to check is that the number of $(j+1)$'s never overtakes the number of $j$'s in the RRW of $A$. 

It is clear that the number of $(j+1)$'s does not surpass the number of $j$'s in the first $j+1$ rows of the tableau. Since each of the remaining rows has at least as many $(j+1)$'s as $j$'s and the last row consists entirely of $(j+1)$'s, the number of $(j+1)$'s can only overtake the number of $j$'s if the total number of $(j+1)$'s in $A$ is greater than the total number of $j$'s in $A$ (that is, if $\mu_{j+1} > \mu_j$). Therefore, it suffices to show that $\mu_j = N_j \geq \mu_{j+1}$. Indeed, observe that $\mu_{j+1} = \sum_{i=j+1}^m \alpha_i - (m-j-2)$, meaning this inequality follows immediately from our assumption that $N_j \geq \sum_{i=j+1}^m \alpha_i - (m-j-2)$. This completes the argument that $A$ (which has shape $\a_{(j \; j+1)}$) is an LR-tableau.

We now show that there is no ribbon LR-tableau of shape $\alpha$ and content $\mu$. By the Yamanouchi property and semistandardness, in any LR-tableau, there cannot be any $(j+1)$'s above the $(j+1)^{st}$ row. It follows that the maximum number of $(j+1)$'s that an LR-tableau of shape $\a$ could have is $\sum_{i=j+1}^m \alpha_i - (m-j-1) < \mu_{j+1}$. Therefore, there is no LR-tableau of shape $\alpha$ and content $\mu$. The LR-rule now tells us that $\mu$ is in the support of $\a_{(j ~ j+1)}$, but is not in the support of $\a$.
\end{proof}

We have proven that the condition in Theorem \ref{strong-nec} is necessary for a ribbon to have full equivalence class. In fact, we have proven that this condition is both necessary and sufficient for ribbons with three or four rows to have full equivalence class \cite{Us}. (When $m=3$, our necessary and sufficient conditions from coincide, while the $m=4$ case requires additional analysis/case-work.) In addition, we have verified by computation that this condition is sufficient for $m=5$, $m=6$, and $m=7$ for certain $n$ (where $n$ is the number of boxes in the diagram). As a result, we conjecture the following:

\begin{conj} \label{strong-nec-conj}
Let $\alpha = (\a_1, \a_2, \ldots, \a_m)$ be a ribbon with $\a_1 \geq \a_2 \geq \cdots \geq \a_m$, where each $\a_i \geq 2$, and $m \geq 3$. Then, $\alpha$ has full equivalence class if and only if $N_j < \sum_{i=j+1}^m \alpha_i - (m-j-2)$ for all $1 \leq j \leq m-2$.
\end{conj}

\begin{remark}
Note that the condition conjectured in Conjecture \ref{strong-nec-conj} to be sufficient for a ribbon to have full equivalence class would subsume the sufficient condition proved in Theorem \ref{suff-real}.
\end{remark}

\section{Concluding Remarks and Future Work}

In this paper, we have presented substantial progress towards classifying when a permutation $\pi \in S_m$ of row lengths of a ribbon $\a$ produces a ribbon $\a_{\pi}$ with the same Schur support as $\a$. However, there are several ways in which we would like to generalize our results so as to obtain a more complete answer to the following central question:
\begin{problem} \label{problem:main}
Given two skew shapes $\lambda_1/\mu_1$ and $\lambda_2/\mu_2$, when is it the case that $[\lambda_1/\mu_1] = [\lambda_2/\mu_2]$?
\end{problem}

To fully understand when ribbons have equal Schur supports, we would first like to prove Conjecture \ref{strong-nec-conj}, which would completely classify when a ribbon has full equivalence class. We would then like to investigate support equalities among ribbons which do not have full equivalence class. Namely, we pose the following open question:
\begin{problem}
Given a  ribbon $\a$, for which $\pi \in S_m$ do we have $[\a] = [\a_{\pi}]$?
\end{problem}
\noindent Theorem \ref{Tri-Swap-Thing} offers partial progress towards answering this question, giving a sufficient condition for the support containment $[\a] \subseteq [\a_{(i \; i+1)}]$ for any $1 \leq i \leq (m-1)$.

Although Problem \ref{problem:main} has proven to be difficult in general, one potentially feasible step forward would be to develop analogues of our main results for skew shapes other than connected ribbons. Namely, we would like to answer the following question:
\begin{problem}
Given a skew shape $\lambda/\mu$, when does $\lambda/\mu$ share a support with every skew shape formed by permuting its row lengths so as to preserve column overlaps (i.e. when does $\lambda/\mu$ have full equivalence class)?
\end{problem}
\noindent Note, however, that a skew shape which is not a ribbon can share Schur support with a skew shape that is not formed by permuting its row lengths. Nonetheless, extending the methods of this paper to general skew shapes has the potential to be illuminating.  

Overall, there are many remaining open questions regarding Schur support equalities. Answering these questions has the potential to better our understanding of the relationship between skew Schur functions and straight Schur functions, two of the most important bases of the symmetric functions.

\section*{Acknowledgements}

This research was performed as part of the 2017 University of Minnesota, Twin Cities Combinatorics REU, and was supported by NSF RTG grant DMS-1148634 and by NSF grant DMS-1351590. We would like to thank Victor Reiner, Pavlo Pylyavskyy, and Galen Dorpalen-Barry for their advice, mentorship, and support. We are also grateful to Sunita Chepuri for her help in finding a reference on $R$-matrices.

\end{document}